\documentclass[leqno,10pt,a4paper]{amsart}
\usepackage[full]{textcomp}
\usepackage{newpxtext}
\usepackage{cabin} 
\usepackage[varqu,varl]{inconsolata} 
\usepackage[bigdelims,vvarbb]{newpxmath} 
\usepackage[cal=cm,bb=esstix,bbscaled=1.05]{mathalfa} 
\usepackage[margin=2.4cm]{geometry}
\usepackage{savesym}
\let\savedegree\bigtimes
\let\bigtimes\relax
\usepackage{mathabx}
\let\bigtimes\savedegree
\usepackage[usenames,dvipsnames]{color}
\usepackage{hyperref}
\hypersetup{
 colorlinks,
 linkcolor={red!50!black},
 citecolor={blue!50!black},
 urlcolor={blue!80!black}
}
\usepackage{tikz}
\usepackage{graphicx}
\usepackage[all,cmtip]{xy}
\usepackage{mleftright}
\usepackage{booktabs}
\usepackage{cite}
\usepackage{enumitem}

\setlist[enumerate]{labelsep=*, leftmargin=1.5pc}
\setlist[enumerate]{label=\normalfont(\roman*), ref=\roman*}
\newtheorem{thm}{Theorem}[section]
\newtheorem{lemma}[thm]{Lemma}
\newtheorem{cor}[thm]{Corollary}

\theoremstyle{definition}
\newtheorem{example}[thm]{Example}
\newtheorem{remark}[thm]{Remark}
\newtheorem{definition}[thm]{Definition}
\newtheorem{conjecture}[thm]{Conjecture}

\numberwithin{equation}{section}
\newcommand{\Z}{\mathbb{Z}}
\newcommand{\Q}{\mathbb{Q}}
\newcommand{\R}{\mathbb{R}}

\newcommand{\N}{\mathbb{N}}
\newcommand{\eps}{\varepsilon}

\renewcommand{\dim}[1]{\operatorname{dim}\mleft({#1}\mright)}

\newcommand{\ph}{\varphi}
\newcommand{\cB}{\mathcal{B}}

\DeclareMathOperator{\Hilb}{Hilb}
\newcommand{\on}{\operatorname}
\newcommand{\ol}{\overline}
\newcommand{\wt}{\widetilde}

\newcommand{\la}{\langle}
\newcommand{\ra}{\rangle}
\newcommand{\pr}{\mathbb{P}}
\graphicspath{{images/}}
\begin{document}
\author[B.\,Wormleighton]{Ben~Wormleighton}
\address{Department of Mathematics\\University of California at Berkeley\\Berkeley, CA\\94720\\USA}
\email{b.wormleighton@berkeley.edu}
\keywords{del Pezzo surface, Hilbert series, orbifold singularity, mirror symmetry}
\subjclass[2010]{14J45 (Primary); 14J17, 13A02 (Secondary)}
\title{Reconstruction of singularities on orbifold del Pezzo surfaces from their Hilbert series}
\maketitle
\begin{abstract}
The Hilbert series of a polarised algebraic variety $(X,D)$ is a powerful invariant that, while it captures some features of the geometry of $(X,D)$ precisely, often cannot recover much information about its singular locus. This work explores the extent to which the Hilbert series of an orbifold del Pezzo surface fails to pin down its singular locus, which provides nonexistence results describing when there are no orbifold del Pezzo surfaces with a given Hilbert series, supplies bounds on the number of singularities on such surfaces, and has applications to the combinatorics of lattice polytopes in the toric case.
\end{abstract}

\section{Introduction}

A fundamental invariant of a polarised algebraic variety $(X,D)$ - that is, a pair consisting of a projective variety $X$ and an ample divisor $D$ on $X$ - is its Hilbert series
$$\Hilb_{(X,D)}(t):=\sum_{d\geq 0}h^0(dD)t^d$$
The Hilbert series is well-known to depend on the embedding supplied by $D$. Suppose that $X$ is a Fano variety, so that its anticanonical divisor $-K_X=\wedge^{\dim X}TX$ is ample. We will always consider Fano varieties to be polarised by their anticanonical divisor and denote
$$\Hilb_X(t):=\Hilb_{(X,-K_X)}(t)$$
A basic question one can pose is how much information about $X$ one can recover from $\Hilb_X(t)$. For example, the dimension and degree of $X$ are readable from the Hilbert polynomial $h_X(d)=h^0(dD)$ of $X$, which can be recovered from its Hilbert series, whereas much information about the singular locus of $X$ cannot be recovered. For example, in this setting the Hilbert series is invariant under $\Q$-Gorenstein deformations. 
\\

This work focuses on del Pezzo surfaces - two dimensional Fano varieties - with a restricted class of isolated singularities and asks what information about the singularities of $X$ can be determined from $\Hilb_X(t)$.
\\

The singularities of interest here are \textit{isolated cyclic quotient surface singularities}, which are given locally by quotients of $\mathbb{A}^2$ by cyclic groups. Let $\mu_r$ denote the group of $r$th roots of unity generated by a primitive $r$th root of unity $\eps$. Write $\frac{1}{r}(a,b)$ for the affine singularity
$$\mathbb{A}^2/\mu_r:=\on{Spec}k[x,y]^{\mu_r}\text{ where $\eps\cdot(x,y)=(\eps^ax,\eps^by)$}$$
This singularity is isolated at the origin if and only if $\on{hcf}(r,a)=\on{hcf}(r,b)=1$. Due to the ambiguity of choice of primitive root of unity, one may as well assume that such a singularity is of the form $\frac{1}{r}(1,a)$. This still doesn't produce a unique representation of the singularity: if $a\ol{a}\equiv1\on{mod}{r}$ then $\frac{1}{r}(1,a)$ is isomorphic to $\frac{1}{r}(1,\ol{a})$ by the change of primitive root given by $\eps\mapsto\eps^{\ol{a}}$.
\\

A del Pezzo surface with only finitely many singularities all of which are isolated quotient cyclic singularities is called an \textit{orbifold del Pezzo surface}. The collection of singularities of an orbifold del Pezzo surface $X$ is termed the \textit{extended basket} of $X$ and denoted by $\wt{\mathcal{B}}_X$. The extended basket of $X$ has no chance of being recovered from the Hilbert series alone because of deformation-invariance. Following Koll\'ar--Shepherd-Barron \cite{ksb}, a cyclic quotient singularity is a $T$\textit{-singularity} if it admits a $\Q$-Gorenstein smoothing. By the work of Akhtar--Kasprzyk \cite{ak}, one can decompose any isolated cyclic quotient singularity into a collection of $T$-singularities and a single \textit{residual singularity} called the \textit{residue} that cannot be smoothed any further. The collection of residues of singularities on $X$ is called the \textit{basket} of $X$ and denoted by $\cB_X$. We will construct a subset of $\cB_X$ called a \textit{reduced basket} for $X$, denoted $\mathcal{RB}$. Subject to two conjectures of combinatorial and number-theoretic nature - Conjectures \ref{conj:1} and \ref{conj:2} - the main results of this paper are:

\begin{thm}[Theorem \ref{thm:main}] Fix a power series $H(t)\in\N\ldbrack t\rdbrack$. Either there are no orbifold del Pezzo surfaces with Hilbert series equal to $H(t)$, or
\begin{itemize}
\item a reduced basket for an orbifold del Pezzo surface with Hilbert series equal to $H(t)$ is one of a finite number of possibilities, which are determined by $H(t)$
\item the extended basket of such an orbifold del Pezzo surface with basket $\mathcal{RB}$ is given by adding a number controlled by $\mathcal{RB}$ and $H(t)$ of singularities obtained from $T$-singularities to the singularities in $\mathcal{RB}$.
\end{itemize}
\end{thm}

It will be made precise below what `obtained from' means and in what sense the number of additional singularities is `controlled' by $\mathcal{RB}$ and $H(t)$.

\begin{thm}[Theorem \ref{thm:deg}] For a given collection of residual singularities $\mathcal{B}$ there exist constants $m>0$ and $M>0$ computable from and dependent only on $\cB$ such that, for any orbifold del Pezzo surface $X$ with $\cB_X=\cB$,
$$m\leq K_X^2\leq M$$
\end{thm}

This second result produces interesting nonexistence results of the following form.

\begin{cor}[Corollary \ref{cor:nonex}] There are no orbifold del Pezzo surfaces with Hilbert series equal to
$$\frac{1+mt+t^2}{(1-t)^3}$$
when $m\geq 11$. There are no toric del Pezzo surfaces with this Hilbert series when $m\geq10$.
\end{cor}

Notice that the Hilbert series of a smooth del Pezzo surface is given by
$$\frac{1+mt+t^2}{(1-t)^3}$$
where $m$ is the degree. Since the $10$ smooth del Pezzo surfaces have degree at most $K_{\mathbb{P}^2}^2=9$ it is clear that there are no smooth del Pezzo surfaces with Hilbert series as in the corollary. The content is that there are also no orbifold del Pezzo surfaces with such a Hilbert series.
\\

In general the Hilbert series of an orbifold del Pezzo surface splits into two parts: an initial term, and a `total orbifold contribution'. Another application of these results is to bound the number of singularities present on an orbifold del Pezzo surface with a given total orbifold contribution and a bound on its Gorenstein index.

\begin{cor}[Corollary \ref{cor:bdd}] Fix a total orbifold contribution $Q\in\Q(t)$ and a positive integer $\ell_*$. There exists a number $N(Q,\ell_*)$ dependent only on and computable from $Q$ and $\ell_*$ such that any orbifold del Pezzo surface with total orbifold contribution $Q$ and Gorenstein index bounded above by $\ell_*$ has at most $N(Q,\ell_*)$ singularities.
\end{cor}

An example of such a result is:

\begin{cor}[Example \ref{ex:l55}] Suppose $X$ is an orbifold del Pezzo surface of Gorenstein index $5\ell$ with total orbifold contribution
$$\frac{2t+t^2+2t^3}{5(1-t^5)}$$
Then $X$ has at most $65\ell+17$ singularities.
\end{cor}

To conclude the introduction, we will briefly describe the context from which these results arose. Residual singularities naturally become objects of interest in the formulation of mirror symmetry for Fano varieties due to \cite{ccgk}. Their mirror construction seeks to produce a possibly singular Fano variety $X$ as a deformation of a toric variety associated to a Landau-Ginzburg model. A Landau-Ginzburg model is simply a Laurent polynomial viewed as a map $f:(\mathbb{C}^\times)^n\to\mathbb{C}$, and one can associate a toric variety $X_f$ to $f$ by taking the toric variety of the face fan of the Newton polytope of monomials in $f$. As $X$ is intended to be a deformation of $X_f$, the singularities of $X$ are related to the singularities of $X_f$: this is where the context of $T$- and residual singularities enters in, which motivated the work of Akhtar--Kapsprzyk from which the present work took its inspiration. For some recent progress in this vein of mirror symmetry, see \cite{lots}.
\\

The proof of the results within this paper will heavily rely on the combinatorial methods of toric geometry because the singularities in question are locally toric. However the results very much apply to nontoric varieties. The structure of the paper is as follows:
\begin{itemize}
\item[\S2-3] Exposition of quotient (or orbifold) singularities and Riemann-Roch for orbifolds
\item[\S4] We study how quotient singularities combinatorially and geometrically decompose via two operations - \textit{shattering} and the \textit{hyperplane sum} - and how these allow the problems at hand to be simplified. The two conjectures that the results of the paper depend upon are stated.
\item[\S5] The main results of the paper are described and established with examples and applications.
\end{itemize}

In a related work \cite{qpc} the techniques developed here are applied to the combinatorial problem of studying lattice points inside rational polytopes.

\subsection*{Acknowledgements} We would like to thank Al Kasprzyk for his inspirational support and collaboration throughout the project. We are also grateful to Alessio Corti, Mohammed Akhtar, Tom Coates, Bernd Sturmfels, and Vivek Shende for many valuable conversations. This research was partially funded by a grant from the London Mathematical Society.

\section{Quotient singularities}

In this section, we will review some of the theory of quotient or orbifold singularities. As a quotient of affine space by an abelian group, the singularity $\frac{1}{r}(1,a)$ is an affine toric variety. Our main reference for toric geometry is \cite{cls} whose notation we follow. Let $N=\mathbb{Z}^2$ and $M=N^\vee:=\on{Hom}_\mathbb{Z}(N,\mathbb{Z})$ be the cocharacter and character lattices respectively of an algebraic $2$-torus $(k^\times)^2$. In toric geometry, a cone $\sigma\subset N_\R:=N\otimes_Z\R$ whose rays are generated by lattice points in $N$ describes an affine toric variety $X_\sigma$. More generally, a collection of cones $\Sigma$ satisfying natural compatibility conditions - usually called a \textit{fan} - produces a non-affine toric variety $X_\Sigma$. The singularity $\frac{1}{r}(1,a)$ is the affine toric variety associated to the cone
$$\sigma_{r,a}=\on{Cone}(e_2,re_1-ae_2)\subset N_\R$$
The lattice height of such a cone - the lattice distance between the origin and the line segment joining the two primitive ray generators of the cone; the \textit{edge} of the cone - is called the \textit{local index} of the cone and can be calculated as in \cite{ak16} to be
$$\ell_{\sigma_{r,a}}=\frac{r}{\on{hcf}(r,a+1)}$$
The width of a cone is the number of lattice line segments along the edge of the cone. Equivalently, it's one less than the number of lattice points along the edge, which can be computed to be
$$\on{width}(\sigma_{r,a})=\on{hcf}(r,a+1)$$
We will often conflate a singularity and its corresponding cone in $N_\R$. Recall that an isolated cyclic quotient singularity is a $T$-singularity if it is $\mathbb{Q}$-Gorenstein smoothable. These were classified by Koll\'ar and Shepherd-Barron:

\begin{lemma}[\cite{ksb}, Prop. 3.10] An isolated cyclic quotient singularity is a $T$-singularity if and only if it takes the form
$$\frac{1}{dn^2}(1,dnc-1)$$
for some $c$ with $\on{hcf}(n,c)=1$.
\end{lemma}

Notice that the local index of a $T$-singularity written as in this classification is $\ell_\sigma=n$ and that there are $dn$ lattice points lying along the edge of the cone. On the level of cones, a $T$-singularity thus looks like:

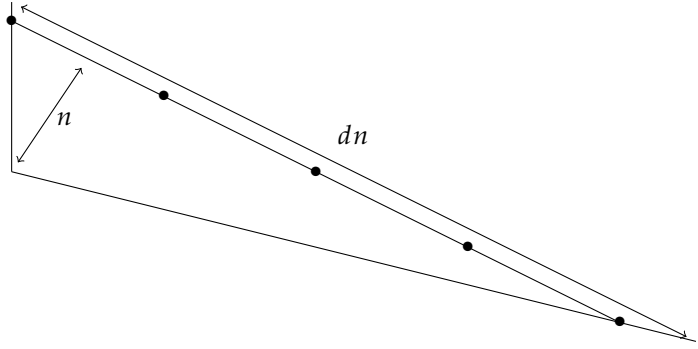
\begin{center}
\begin{tikzpicture}
\node (o) at (0,0){};
\node (e2) at (0,2){$\bullet$};
\node (e2l) at (0,2.25){};
\node (v1) at (8,-2){$\bullet$};
\node (v1l) at (9,-2.25){};
\node (v2) at (6,-1){$\bullet$};
\node (v3) at (4,0){$\bullet$};
\node (v4) at (2,1){$\bullet$};
\node (vmid) at (1,1.5){};
\node (l1) at (0.7,0.7){$n$};
\node (l1) at (4.5,0.5){$dn$};

\draw (o.center) to (e2l.center);
\draw (o.center) to (v1l.center);
\draw (e2.center) to (v1.center);

\draw[<->] (o) to (vmid);
\draw[<->] (e2l) to (v1l);
\end{tikzpicture}
\end{center}

Prosaically, the cone of a $T$-singularity has width a multiple of its height. A $T$-singularity of width equal to local index - as thin as possible -  is called an elementary $T$-singularity.
\\

Crucial to the relevant formulation of mirror symmetry is the notion of polytope mutation developed in \cite{acgk}. This is a combinatorial procedure that produces from a polytope $P$, a `weight vector' $w\in M$, and an appropriate choice of facet $F\subset P$ another polytope $\mu_{w,F} P$ whose toric variety is a deformation of the toric variety associated to $P$. Applied locally to cones, it enables one to realise the smoothing of Koll\'ar and Shepherd-Barron for $T$-singularities as follows. The details of mutation can be found in \cite{acgk} \S3.
\\

Mutating the cone $\sigma_{dn^2,dnc-1}$ for $\frac{1}{dn^2}(1,dnc-1)$ in the facet given by the cone edge and with respect to the weight vector given by the inward normal to the cone edge has the effect of removing a line segment of length $\ell_\sigma=n$ from the edge of $\sigma_{dn^2,dnc-1}$.

\begin{center}
\begin{tikzpicture}[scale=0.6]
\node (o) at (0,0){};
\node (e2) at (0,2){$\bullet$};
\node (e2l) at (0,2.25){};
\node (v1) at (8,-2){$\bullet$};
\node (v1l) at (9,-2.25){};
\node (v2) at (6,-1){$\bullet$};
\node (v3) at (4,0){$\bullet$};
\node (v4) at (2,1){$\bullet$};
\node (vmid) at (1,1.5){};
\node (l1) at (0.8,0.7){$n$};
\node (l2) at (4.5,0.5){$dn$};

\draw (o.center) to (e2l.center);
\draw (o.center) to (v1l.center);
\draw (e2.center) to (v1.center);

\draw[<->] (o) to (vmid);
\draw[<->] (e2l) to (v1l);

\node (o) at (10,0){};
\node (e2) at (10,2){$\bullet$};
\node (e2l) at (10,2.25){};
\node (v1) at (18,-2){$\bullet$};
\node (v1l) at (18.25,-1.8){};
\node (v2) at (16,-1){$\bullet$};
\node (v3) at (14,0){$\bullet$};
\node (v3l) at (14.5,0){};
\node (v4) at (12,1){$\bullet$};
\node (vmid) at (11,1.5){};
\node (l1) at (10.8,0.7){$n$};
\node (l2) at (13.5,1.5){$(d-1)n$};
\node (l3) at (16.5,-0.5){$n$};

\draw (o.center) to (e2l.center);
\draw (o.center) to (v3l.center);
\draw (e2.center) to (v3.center);

\draw[<->] (o) to (vmid);
\draw[<->] (e2l) to (v3l);
\draw[<->] (v3l) to (v1l);
\end{tikzpicture}
\end{center}

The result is that $\sigma_{dn^2,dnc-1}$ mutates to $\sigma_{(d-1)n^2,(d-1)nc-1}$. Successive mutation reduces it to the empty cone; in other words, the singularity is smoothed. The same procedure applies to a general singularity $\frac{1}{r}(1,a)$ yet may not terminate in the empty cone. Indeed, by definition the only cones for which this process will reduce to the empty cone are the $T$-singularities. Let $\sigma_{r,a}$ be the cone for $\frac{1}{r}(1,a)$ and let its local index be $\ell$. In general, one will be able to remove subcones of the form $\sigma_{d\ell^2,d\ell c-1}$ from $\sigma_{r,a}$ by mutation and hence can deform to a cone with width equal to the residue of the width of $\sigma_{r,a}$; in particular, smaller than $\ell$. This cone is the \textit{residue} $\on{res}(\sigma_{r,a})$ of $\sigma_{r,a}$. A singularity is called \textit{residual} if it is an isolated cyclic quotient singularity with width less than its local index. In other words, no more line segments can be removed from its edge by mutation.

\begin{lemma} An isolated cyclic quotient singularity is a residual singularity if and only if it can be written in the form
$$\frac{1}{k\ell}(1,kc-1)$$
where $\ell$ is the local index, $0<k<\ell$, $\on{hcf}(\ell,c)=1$, and $\on{hcf}(\ell,kc-1)=1$.
\end{lemma}

\begin{proof} Let $\frac{1}{r}(1,a)$ be a residual singularity of local index $\ell$. Then, letting $k=\on{hcf}(r,a+1)$, $r=k\ell$ and $a+1=kc$ for some $c$ with $\on{hcf}(\ell,c)=1$. The final coprimality condition is necessary for the singularity to be isolated. The width of $\frac{1}{r}(1,a)$ is $k$ and so the condition $0<k<\ell$ is equivalent to the singularity being residual.
\end{proof}

The quantity $c$ is called the \textit{slope} of the singularity (or of the cone).

\section{Hilbert series}

Recall that the [anticanonical] Hilbert series of a Fano variety $X$ is defined to be
$$\Hilb_X(t):=\sum_{d\geq0}h^0(-dK_X)t^d$$
In \cite{ak}, Ahktar--Kasprzyk use previous work of Reid in \cite{ypg} to produce the following formula for the Hilbert series of an orbifold del Pezzo surface.

\begin{lemma}[\cite{ak}, Cor. 3.5]\label{lem:hilb} Suppose $X$ is an orbifold del Pezzo surface with basket $\cB$. Then
$$\Hilb_X(t)=\frac{1+(K_X^2-2)t+t^2}{(1-t)^3}+\sum_{\sigma\in\cB}Q_\sigma$$
where
$$Q_{\sigma_{r,a}}=\frac{1}{1-t^r}\sum_{i=1}^r(\delta_{r,a,(a+1)i}-\delta_{r,a,0})t^{i-1}$$
where $\delta_{r,a,i}$ is a \textnormal{Dedekind sum} defined as
$$\delta_{r,a,i}=\frac{1}{r}\sum_{\xi\in\mu_r\setminus\{1\}}\frac{\xi^i}{(1-\xi)(1-\xi^a)}$$
\end{lemma}

The sum ranges over the group $\mu_r$ of $r$th roots of unity except $1$. Dedekind sums are intriguing quantities much studied in number theory and combinatorics - for example, \cite{z73} and the survey \cite{br} - that frequently recur in studying the cohomology of toric varieties as in \cite{p93}. Notice that each singularity $\sigma$ in $\cB$ makes two contributions to the Hilbert series: a local contribution in the orbifold correction terms $Q_\sigma$ and a global contribution to the degree. The latter is captured by...

\begin{lemma}[\cite{ak}, Prop. 3.3]\label{lem:degf} The degree of an orbifold del Pezzo surface with basket $\mathcal{B}$ is given by
$$K_X^2=12-n-\sum_{\sigma\in\cB}A_\sigma$$
where $n$ is the Euler number of the smooth locus of $X$, and
$$A_{\sigma_{r,a}}=m+1-\sum_{i=1}^md_i^2b_i + 2\sum_{i=1}^md_id_{i+1}$$
with $[b_1,\dots,b_m]$ the Hirzebruch-Jung continued fraction expansion of $\frac{r}{a+1}$ and $d_i$ the discrepancy at the $i$th component of the exceptional fibre of the minimal resolution of $\sigma_{r,a}$.
\end{lemma}

\begin{example}\label{ex:T} A $T$-singularity $\tau$ of the form $\frac{1}{dn^2}(1,dnc-1)$ has $A_\tau=d$ and $Q_\tau=0$.
\end{example}

It is worth noting that these formulae are explicit enough to actually compute all the possible combinations of singularities that would produce a given power series. However, they allow no structural description of this collection of possible baskets, which turns out to be very geometric.

\section{Shattering and the hyperplane sum}

Much of the present work is geared towards finding and investigating high amounts of structure in the collection of quotient singularities, especially at a fixed local index. We introduce two inverse operations - shattering and the hyperplane sum - that decompose or combine singularities of a fixed local index. The singularities that are indecomposable with respect to these operations contain the cohomological information of all quotient singularities and hence allow the problems considered here to be simplified.

\subsection{Shattering}

Suppose $\sigma$ is a cone and that $v$ is a primitive lattice point on the edge of $\sigma$. The fan obtained by inserting a ray through $v$ and hence dividing $\sigma$ into two new cones is a crepant blowup $Y$ of $X_\sigma$. Recall that a map $f:Y\to X$ of Gorenstein schemes is crepant if $f^*K_X=K_Y$; that is, it has no discrepancy. In the toric situation above this occurs exactly when the ray generator $v$ lies on the edge of $\sigma$. This toric variety $Y$ has two affine pieces corresponding to the two dimensional cones, which are each cyclic quotient singularities. The local indices of these new singularities will be the same as the local index of the original cone since $v$ was primitive. This procedure can be repeated to produce a collection of cones with edges lying along a common hyperplane. We call this operation \textit{shattering}, and the cones obtained from $\sigma$ by this process \textit{shards} of $\sigma$. Below we show $\frac{1}{9}(1,2)$ shattering into two shards corresponding to the singularities $\frac{1}{6}(1,1)$ and $\frac{1}{3}(1,1)$.

\begin{center}
\begin{tikzpicture}[scale=0.6]
\node (o) at (0,0){};
\node (e2) at (0,2){$\bullet$};
\node (e2l) at (0,2.25){};
\node (v2l) at (6.5,-1.05){};
\node (v2) at (6,-1){$\bullet$};
\node (v3) at (4,0){$\bullet$};
\node (v4) at (2,1){$\bullet$};
\node (v5) at (1,0.5){$\bullet$};

\draw (o.center) to (e2l.center);
\draw (o.center) to (v2l.center);
\draw (e2.center) to (v2.center);

\node (o) at (10,0){};
\node (e2) at (10,2){$\bullet$};
\node (e2l) at (10,2.25){};
\node (v2l) at (16.5,-1.05){};
\node (v2) at (16,-1){$\bullet$};
\node (v3) at (14,0){$\bullet$};
\node (v3l) at (14.5,0){};
\node (v4) at (12,1){$\bullet$};
\node (v5) at (11,0.5){$\bullet$};
\node (vmid) at (11,1.5){};

\draw (o.center) to (e2l.center);
\draw (o.center) to (v3l.center);
\draw (e2.center) to (v3.center);
\draw (o.center) to (v2l.center);
\draw (v3.center) to (v2.center);

\node (l1) at (3.5,1.5){$\frac{1}{9}(1,1)$};
\node (l2) at (13.25,1.6){$\frac{1}{6}(1,1)$};
\node (l1) at (16,0){$\frac{1}{3}(1,1)$};
\end{tikzpicture}
\end{center}

Note that the blowup corresponding to the ray through the other lattice point in the interior of the cone edge would not be crepant as this lattice point is not primitive.

\subsection{The hyperplane sum}

\begin{definition}\label{def:hyp} The \textit{hyperplane sum} of two cones $\sigma_1=\on{Cone}(u,v)$ and $\sigma_2=\on{Cone}(v,w)$ with ray generators clockwise ordered such that:
\begin{itemize}
\item the vectors $v-u$ and $w-v$ are parallel,
\item both cones are of the same local index,
\end{itemize}
to be $\sigma_1*\sigma_2:=\on{Cone}(u,w)$, which is the usual Minkowski sum when it is defined.
\end{definition}
Equivalently $\sigma_1$ and $\sigma_2$ have hyperplane sum defined to be $\sigma_1+\sigma_2$ only when this cone also has the same local index as $\sigma_1$ and $\sigma_2$. One can see that this process is inverse to shattering. Continuing the conflation of cones and singularities, we will say that the hyperplane sum $\frac{1}{r}(1,a)*\frac{1}{s}(1,b)$ of two singularities is defined if there are two cones $\sigma_1$ and $\sigma_2$ corresponding to the singularities that are as in Definition \ref{def:hyp}. In this case, the hyperplane sum is defined to be the singularity corresponding to the resulting cone $\sigma_1*\sigma_2$. Thus,
$$\frac{1}{6}(1,1)*\frac{1}{3}(1,1)=\frac{1}{9}(1,2)$$
Notice that this is well-defined as the singularity defined by a cone is unchanged by the action of $\on{GL}_2(\mathbb{Z})$ on $N$ but that it is noncommutative: the hyperplane sum $\sigma_2*\sigma_1$ in general won't even be defined if $\sigma_1*\sigma_2$ is.
\\

Geometrically, the hyperplane sum is a crepant blowdown contracting the torus-invariant curve corresponding to the ray through $v$. From this is follows that...

\begin{lemma}\label{lem:add} (Additivity) Let $\sigma_1*...*\sigma_n=\sigma$. Then
$$Q_{\sigma_1}+...+Q_{\sigma_n}=Q_\sigma\text{ and }A_{\sigma_1}+...+A_{\sigma_n}=A_\sigma$$
\end{lemma}

\begin{cor}\label{cor:can} Let $\sigma_1*...*\sigma_n=\tau$, a $T$-singularity. Then
$$Q_{\sigma_1}+...+Q_{\sigma_n}=0\text{ and }A_{\sigma_1}+...+A_{\sigma_n}=A_\tau=\frac{\on{width}(\tau)}{\ell_\tau}$$
\end{cor}

\begin{proof}[Proof of Lemma \ref{lem:add}] Let $X$ be a toric del Pezzo surface whose fan contains the cone $\sigma$ and with all other cones smooth. Let $Y$ be the toric variety associated to the fan that agrees with $X$ everywhere outside of $\sigma$ and there has $\sigma$ replaced by $\sigma_1,\dots,\sigma_n$. The natural map $f:Y\to X$ is given by $n-1$ crepant blowups in torus-fixed points according to the cones $\sigma_1,\dots,\sigma_n$. Since $f$ is surjective, $h^0(-dK_X)=h^0(-d(f^*K_X))=h^0(-dK_Y)$. Thus the Hilbert series of $X$ agrees with that of $Y$.
\\

Note that $K_X^2=f^*K_X^2=K_Y^2$ and so the initial term of $\Hilb_X(t)$ and $\Hilb_Y(t)$ agree. Hence the orbifold correction terms $Q_\varsigma$ satisfy
$$Q_\sigma=\sum_{\varsigma\in\cB_X}Q_\varsigma=\sum_{\varsigma\in\cB_Y}Q_\varsigma=\sum_{i=1}^nQ_{\sigma_i}$$
as desired. Suppose that there are $m$ smooth cones in the fans of $X$ and $Y$. The degree formula above gives that
$$12-m-A_\sigma=K_X^2=K_Y^2=12-m-\sum_{i=1}^nA_{\sigma_i}$$
and so the degree correction terms also agree. The corollaries follow from the lemma combined with Example \ref{ex:T}.
\end{proof}

In the sense of Corollary \ref{cor:can} $T$-singularities are negligible from the perspective of orbifold contributions to Hilbert series, though they still contribute to the degree. Say that a residual singularity $\sigma$ is \textit{hyperplane inverse} (or just `inverse' if the context is clear) to another residual singularity $\sigma'$ if the hyperplane sum $\sigma*\sigma'$ is defined and equal to an elementary $T$-singularity. By explicit calculation one finds that:

\begin{lemma}\label{lem:inv} The hyperplane inverse of $\sigma:\frac{1}{kn}(1,kc-1)$ is
$$\sigma^{-1}:\frac{1}{n(n-k)}(1,(n-k)(n-c)-1).$$
\end{lemma}

Notice that $(\sigma^{-1})^{-1}$ is isomorphic to $\sigma$. Let the `dual singularity' $\overline{\frac{1}{r}(1,a)}:=\frac{1}{r}(1,\ol{a})$ with $a\ol{a}\equiv 1\on{mod}{r}$ and call a cyclic quotient singularity $\sigma$ \textit{self-dual} if $\ol{\sigma}=\sigma$. The effect of dualising a singularity reverses its gluing behaviour in that two singularities $\sigma$ and $\sigma'$ are hyperplane summable if and only if $\ol{\sigma}'$ and $\ol{\sigma}$ are. This follows as the element of $\on{GL}_2(\Z)$ taking $\sigma$ to $\ol{\sigma}$ is orientation-reversing. The hyperplane sum is thus not actually defined on isomorphism classes of quotient singularities, but requires a choice of root of unity in addition.

\subsection{Indecomposable singularities and maximal shatterings}

Viewed on the level of cones, one can split a given residual singularity into `indecomposable' hyperplane summands: those that cannot be shattered further. These are exactly the cones containing no primitive lattice points along the interior of their edge, which are not necessarily of width $1$. 

\begin{definition} A cyclic quotient singularity $\sigma$ is \textit{indecomposable} if a (or every) cone corresponding to it has no primitive lattice points lying on the interior of its edge.
\end{definition}

Note that an indecomposable singularity is in particular a residual singularity of local index at least $3$. Define the \textit{residual quiver} at local index $\ell$ to be the quiver with a vertex for every indecomposable singularity (distinguishing dual singularities) of local index $\ell$ and with an arrow drawn between $\sigma$ and $\sigma'$ if and only if $\sigma$ and $\sigma'$ are hyperplane summable. The following quiver is the result of applying the construction in the case of local index $5$.
\\

\hspace{2in}
\begin{tikzpicture}[scale=1]
\node (a) at (-5,3) {$\frac{1}{5}(1,1)$};
\node (b) at (-2,3) {$\frac{1}{5}(1,2)$};
\node (c) at (-2,1) {$\frac{1}{10}(1,1)$};
\node (e) at (-5,1) {$\frac{1}{5}(1,3)$};

\draw[->] (a)--(b);
\draw[->] (b)--(c);
\draw[->] (e)--(a);
\draw[->] (c)--(e);
\end{tikzpicture}

\begin{lemma} For given $\ell$ the residual quiver at local index $\ell$ is a cycle of length $\phi(\ell)$, where $\phi$ is Euler's totient function. Moreover, there is exactly one indecomposable singularity of every slope $c\in(\mathbb{Z}/\ell\mathbb{Z})^\times$.
\end{lemma}

Before proving this lemma, observe that for any two singularities $\sigma_1$ and $\sigma_2$, though they may not be hyperplane summable, there will always be a unique singularity $\sigma_g$ of smallest width such that $\sigma_1*\sigma_g$ and $\sigma_g*\sigma_2$ are well-defined. This is because, picking a cone for $\sigma_1$, one can always move a cone representing $\sigma_2$ to have edge lying on the hyperplane given by the cone edge of $\sigma_1$ by the action of $\on{GL}_2(\Z)$. This cone $\sigma_g$ is called the \textit{gluing cone} of $\sigma_1$ and $\sigma_2$.

\begin{proof} If $\on{hcf}(\ell,c-1)=1$ then the singularity of slope $c$ and width $1$ is indecomposable. Suppose $\on{hcf}(\ell,c-1)\not=1$ but $\on{hcf}(\ell,2c-1)=1$. Then the residual singularity with slope $c$ and width $2$ contains only non-primitive interior edge lattice points and so is also indecomposable. Continuing, if $\on{hcf}(\ell,c-1)\not=1$ and $\on{hcf}(\ell,2c-1)\not=1$ but $\on{hcf}(\ell,3c-1)=1$ then the residual singularity with slope $c$ and width $3$ will be indecomposable. For a given slope $c$ one can pick $w=2c^{-1}\in\mathbb{Z}/\ell\mathbb{Z}$ as a width for which $\on{hcf}(\ell,wc-1)=1$ since $wc\equiv 1\on{mod}{\ell}$. Hence this process exhausts all $c\in(\mathbb{Z}/\ell\mathbb{Z})^\times$. Clearly no other values of $c$ are possible and each level with fixed width excludes all previous levels, giving that the number of vertices of the residual quiver is equal to the number of choices of $c$, which is $\phi(\ell)$. There is a unique indecomposable singularity that can be glued onto a given singularity by extending its cone edge. It follows that every vertex is linked to exactly two others. This says that the quiver is a collection of cycles. However given any two indecompables $\sigma_1$ and $\sigma_2$ there is a gluing cone joining them. Maximally shattering this gluing cone gives a path from $\sigma_1$ to $\sigma_2$ in the residual quiver; therefore the quiver is connected and so forms a single cycle.
\end{proof}

\subsection{Conjecture \#1}

In this section we will state the first conjecture that the main results of this paper will depend upon. We will fix notation:
\begin{itemize}
\item $\mathcal{R}_\ell$ is the set of indecomposable singularities of local index $\ell$
\item $\on{Res}(\ell)$ is the set of residual singularities of local index $\ell$
\item $\sigma$ will denote a residual singularity of local index $\ell$
\end{itemize}
From the general theory of Hilbert series for Gorenstein schemes - see, for example, \cite{bass}, \cite{buck}, or \cite{mnz} - applied to this particular situation one can write
\begin{equation} \tag{$\star$}
Q_\sigma=\frac{\delta_0+\delta_1t+...+\delta_{\ell-1}t^{\ell-1}}{\ell(1-t^\ell)}
\end{equation}
for $\delta_i\in\Z$. Define the $\delta$\textit{-vector} of $\sigma$ to be $\delta(\sigma):=(\delta_0,\delta_1,...,\delta_{\ell-1})$. From Lemma \ref{lem:hilb}
$$\delta_i=\ell(\delta_{r,a,(a+1)(i+1)}-\delta_{r,a,0})$$
The $\delta$-vector has the properties:
\begin{itemize}
\item $\delta(\sigma)$ is palindromic
\item $\delta_0=\delta_{\ell-1}=0$.
\end{itemize}
Because of the second property, we will abbreviate the $\delta$-vector to omit the first and last terms. For example, the $\delta$-vector of $\frac{1}{5}(1,1)$ with
$$Q_{\frac{1}{5}(1,1)}=\frac{t^3-2t^2+t}{5(1-t^5)}$$
is $(1,-2,1)$. We now prove these two properties.

\begin{proof} Suppose $\sigma=\frac{1}{r}(1,a)$ with local index $\ell$. We start by proving that $Q_\sigma$ can be written in the form $(\star)$. It suffices that the numerator in the expression of $Q_\sigma$ in Lemma \ref{lem:hilb} with denominator $1-t^r$ is divisible by $1+t^\ell+t^{2\ell}+\dots+t^{r-\ell}$, or equivalently that
$$\delta_{r,a,(a+1)(\ell+i)}=\delta_{r,a,(a+1)i}$$
This follows immediately from noting that $(a+1)\ell\equiv0\on{mod}{r}$ and that the final argument of such a Dedekind sum is well-defined modulo $r$. To prove the first property - that $\delta_{\ell-1-i}=\delta_i$ - observe that it suffices that the first $\ell$ terms of the numerator of $Q_\sigma$ in the expression of Lemma \ref{lem:hilb} are palindromic, or that
$$\delta_{r,a,(a+1)(\ell-i)}=\delta_{r,a,(a+1)(i+1)}$$
Computing directly using the fact that $(a+1)\ell\equiv0\on{mod}{r}$,
$$\delta_{r,a,(a+1)(\ell-i)}=\sum\frac{\eps^{(a+1)(\ell-i)}}{(1-\eps)(1-\eps^a)}=\sum\frac{\eps^{-(a+1)i}}{(1-\eps)(1-\eps^a)}$$
Multiplying by $\eps^{-(a+1)}$ in numerator and denominator yields
$$\sum\frac{\eps^{-(a+1)(i+1)}}{\eps^{-(a+1)}(1-\eps)(1-\eps^a)}=\sum\frac{\eps^{-(a+1)(i+1)}}{(\eps^{-1}-1)(\eps^{-a}-1)}=\sum\frac{\eps^{(a+1)(i+1)}}{(1-\eps)(1-\eps^a)}=\delta_{r,a,(a+1)(i+1)}$$
using the bijection $\eps\mapsto\eps^{-1}$ on the $r$th roots of unity. The second property now follows from the equality
$$\delta_{r,a,(a+1)\ell}=\delta_{r,a,0}$$
since $(a+1)\ell\equiv0\on{mod}{r}$ and so the $(\ell-1)$th coefficient $\delta_{\ell-1}=\delta_{r,a,(a+1)\ell}-\delta_{r,a,0}=0$.
\end{proof}

\begin{lemma}\label{lem:tvan} An isolated cyclic quotient singularity $\tau$ is a $T$-singularity if and only if $Q_\tau=0$.
\end{lemma}

\begin{proof} The only if implication follows from Example \ref{ex:T}. It suffices to show that every residual singularity makes a nonzero contribution to the Hilbert series. This follows using the shattering in \cite{ak} decomposing a cone into $T$-cones and a single residual cone. For the forward implication, let $\frac{1}{r}(1,a)$ be a residual singularity and consider the weighted projective plane $X=\pr(1,a,r)$. This has three affine pieces isomorphic to $\mathbb{A}^2,\frac{1}{r}(1,a)$ and $\frac{1}{a}(1,r)$ and thus its Hilbert series is
$$\on{Hilb}_X(t)=\frac{1+(K_X^2-2)t+t^2}{(1-t)^3}+Q_{\sigma_1}+Q_{\sigma_2}$$
where $\sigma_1:\frac{1}{r}(1,a)$ and $\sigma_2:\frac{1}{a}(1,r)$. Let these have local indices $\ell_1$ and $\ell_2$ - which are coprime - and write
$$\delta(\sigma_i)=(\delta_j^i)_{j=1,\dots,\ell_i-2}$$
One can compute the $t$-coefficient of the Hilbert series to be
\begin{equation} \tag{$\dagger$}
h^0(-K_X)=1+K_X^2+\frac{\delta_1^1}{\ell_1}+\frac{\delta_1^2}{\ell_2}
\end{equation}
As the dimension of a vector space, this must be an integer. Recall that the degree of $X$ is $(1+a+r)^2/ar$, which has the same fractional part as
\begin{equation}\tag{$\ast$}
\frac{1+a}{r}+\frac{1+r}{a}+\frac{1+a+r}{ar}
\end{equation}
Consider the the residues of $(*)$ mod $\Z\cdot\frac{1}{\ell_1}$ and $\Z\cdot\frac{1}{\ell_2}$. Suppose the residue of $(*)$ mod $\Z\cdot\frac{1}{\ell_2}$ is zero. Then
$$\frac{1+a}{r}+\frac{1+a+r}{ar}\equiv0\on{mod}\Z\cdot\frac{1}{\ell_2}$$
as $(1+r)/a$ has denominator $\ell_2$ in lowest terms. Combining fractions, this requires in particular that $r$ divides $(1+a)^2$ as $r$ is coprime to $\ell_2$. Let $k$ be the width of $\sigma_1$ so that $r=k\ell_1$ and $1+a=kc$ for some $c$ coprime to $\ell_1$. For $k\ell_1$ to divide $k^2c^2$ one must have that $\ell_1$ divides $k$, but this is contrary to the definition of residual singularity. It follows that both residues are nonzero and so, for $(\dagger)$ to be an integer, $\delta_1^1$ and $\delta_1^2$ must be nonzero.
\end{proof}

This proof actually shows that the first coefficient of the $\delta$-vector of a residual singularity is nonzero. Notice that this argument would fail for a $T$-singularity where, by definition, the local index $\ell_1$ divides the width $k$.
\\

Given a semigroup $S$ and a set $R$, let the formal semigroup consisting of $S$-linear combinations of elements of $R$ be denoted by $S\langle R\rangle$. If $S=\Z$ this is just the formal lattice generated by $R$. Define the $\delta$\textit{-lattice} for local index $\ell$ to be the sublattice
$$\Delta(\ell):=\mathbb{Z}\langle\delta(\sigma):\sigma\in\on{Res}(\ell)\rangle\subset\Z^{\ell-2}$$
generated by all the $\delta$-vectors of residual singularities (equivalently, indecomposable singularities) of local index $\ell$.
\\

Given a list of residuals $\mathcal{T}=(\sigma_1,\dots,\sigma_n)$ there is a unique expression $(\sigma^1_1,\dots,\sigma_1^{m_1},\dots,\sigma_n^1,\dots,\sigma_n^{m_n})$ where $\sigma_i^j\in\mathcal{R}_\ell$ and $\sigma_i^1*\dots*\sigma_i^{m_i}=\sigma_i$. This tuple is the \textit{maximal shattering} of $\mathcal{T}$ denoted by $\rho(\mathcal{T})$. Combinatorially the cones corresponding to each $\sigma_i$ have been shattered as much as possible to decompose into a hyperplane sum in terms of $\mathcal{R}_\ell$. Note that by definition of $\mathcal{R}_\ell$ this is maximal exactly in this sense. Of course
$$\sum_{i}{Q_{\sigma_i}}=\sum_{i,j}{Q_{\sigma_i^j}}\text{ and }\sum_{i}{A_{\sigma_i}}=\sum_{i,j}{A_{\sigma_i^j}}$$
by the additivity in Lemma \ref{lem:add}. From the characterisation in terms of non-primitive lattice points, the maximal shattering of any singularity is unique. After being linearly extended $\rho$ defines a surjective monoid homomorphism $\mathbb{N}[\on{Res}(\ell)]\to\mathbb{N}[\mathcal{R}_\ell]$ which is left-inverse to the inclusion $\mathbb{N}[\mathcal{R}_\ell]\to\mathbb{N}[\on{Res}(\ell)]$. Consider the map $\wt{\Phi}_\ell$ completing the diagram
$$\xymatrix{\mathbb{N}[\on{Res}(\ell)]\ar[rd]^{\Phi_\ell} \ar[d]_{\rho} & \\
\mathbb{N}[\mathcal{R}_\ell] \ar[r]_{\wt{\Phi}_\ell} & \Delta(\ell)}$$
where $\Delta(\ell)$ is the lattice of $\delta$-vectors of orbifold contributions of local index $\ell$ as above. $\wt{\Phi}_\ell$ exists and is unique since, geometrically, $\rho$ applies a collection of crepant blowups that preserve the orbifold contributions $Q_\sigma$ and hence their $\delta$-vectors. We make use of semigroups of the form $\N\langle R\rangle$ to record the (nonnegative) quantities of each singularity inside a basket. These maps all extend to lattice homomorphisms
$$\Phi_\ell^\Z:\Z\la\on{Res}(\ell)\ra\to\Delta(\ell),\wt{\Phi}_\ell^\Z:\Z\la\mathcal{R}_\ell\ra\to\Delta(\ell),\rho^\Z:\Z\la\on{Res}(\ell)\ra\to\Z\la\mathcal{R}_\ell\ra$$
With the objective of studying relations between orbifold contributions we assume the following conjecture, which has been verified up to local index $34$ in \textsc{Sage}. It has echoes of the $\frac{1}{2}\phi(r)$ in \cite{ypg} \S5.9 as well as of many other results across the study of Dedekind sums. The reader can add the caviate $\ell\leq 34$ on any subsequent results making use of this conjecture.

\begin{conjecture}\label{conj:1} $\on{rank}{\Delta(\ell)}=\frac{1}{2}\phi(\ell)$.
\end{conjecture}

\subsection{Cancelling tuples}

Note that any collection of singularities in the kernel of $\Phi_\ell$ contributes zero in orbifold correction terms to the Hilbert series.

\begin{definition} A \textit{cancelling tuple} is a finite collection of residual singularities $\sigma_1,...,\sigma_n$ such that $\sum_{i=1}^nQ_{\sigma_i}=0$. 
\end{definition}

Equivalently, $\sum_{i=1}^n\sigma_i\in\on{ker}{\Phi_\ell}$. If $\sigma_1,...,\sigma_n$ is a cancelling tuple then $\rho(\sigma_1),...,\rho(\sigma_n)$ is also a cancelling tuple. Hence $\sum_{i=1}^n\rho(\sigma_i)\in\on{ker}{\wt{\Phi}_\ell}$. A cancelling tuple can be decomposed nonuniquely into \textit{minimal} cancelling tuples: cancelling tuples that contain no smaller cancelling tuples.

\begin{lemma} An elementary $T$-singularity of local index $\ell$ is composed of exactly one of every indecomposable at local index $\ell$ glued in the cyclic order prescribed by the residual quiver. Hence, distinguishing dual singularities, there are exactly $\phi(\ell)$ $T$-singularities at local index $\ell$ parameterised by choosing a starting point in the residual quiver.
\end{lemma}

\begin{proof} Write $\tau=\sigma_1*\dots*\sigma_m$ with each $\sigma_i$ an indecomposable singularity. Since a $T$-singularity can be glued to itself by mutation, $\sigma_m$ must be the indecomposable singularity immediately preceding $\sigma_1$ in the residual quiver. Thus, since one must follow the cycle around the quiver in order to glue indecomposable singularities, $\tau$ is a circuit of the quiver beginning at $\sigma_1$ and ending at the previous vertex $\sigma_m$. Suppose $\sigma_1$ appears again as one of the $\sigma_i$ for $i>1$. Then $\tau=(\sigma_1*...*\sigma_m)*(\sigma_1*...*\sigma_m)*...*(\sigma_1*...*\sigma_m)$ so that the cycle can end with $\sigma_m$. Suppose there are $p$ cycles of the residual quiver in this decomposition of $\tau$. The cone $\sigma_1*...*\sigma_m$ has $p\sum_{i=1}^mQ_{\sigma_i}=Q_\tau=0$ and so $\sum_{i=1}^mQ_{\sigma_i}=0$. Thus by Lemma \ref{lem:tvan} $\sigma_1*\dots*\sigma_m$ is also a $T$-singularity. This contradicts the fact that $\tau$ is an elementary $T$-singularity unless there is only a single circuit of the quiver.
\end{proof}

\begin{cor} The widths of all indecomposable singularities of local index $\ell$ sum to $\ell$.
\end{cor}

Let $\sigma_1,...,\sigma_n$ be the indecomposable singularities of local index $\ell$ listed in cyclic order according to the residual quiver. From the lemma above $\Phi_\ell(\sum_{i=1}^n\sigma_i)=0$ and so $\sigma_1,...,\sigma_n$ form a cancelling tuple. This arises from maximally shattering an elementary $T$-cone. Hence, $\on{rank}\on{ker}\wt{\Phi}^\Z_\ell\geq1$ as there is at least one cancelling tuple of every local index, which is formed of indecomposable singularities. The objective of the rest of this section is to prove, assuming the conjecture on the rank of $\Delta(\ell)$, that:

\begin{lemma}\label{lem:can} All of the minimal cancelling tuples consisting of singularities of local index $\ell$ arise from shattering a $T$-cone in some way.
\end{lemma}

Clearly shattering a $T$-cone does produce a cancelling tuple, however the converse is more subtle.
\\

$\wt{\Phi}_\ell$ surjects onto the $\delta$-lattice $\Delta(\ell)\cong\mathbb{Z}^{\frac{1}{2}\phi(\ell)}$ as $\Phi_\ell$ does and so $\on{rank}{\wt{\Phi}_\ell^\Z}=\frac{1}{2}\phi(\ell)$. We can identify isomorphic singularities in $\mathcal{R}_\ell$: this has the effect of conflating the singularities $\sigma$ and $\bar{\sigma}$. These two singularities have the same orbifold contribution and so $\wt{\Phi}_\ell$ passes to the quotient to give a surjection $\wt{\Phi}_\ell^\Z:\mathbb{Z}\la\mathcal{R}_\ell/{\cong}\ra\to\Delta(\ell)$. The rank of $\mathbb{Z}\la\mathcal{R}/{\cong}\ra$ is $\phi(\ell)-\frac{1}{2}S(\ell)$ where $S(\ell)$ is the number of non-self dual residuals of local index $\ell$ contained in the generating set $\mathcal{R}_\ell$, which is seen by noting that one of each pair $\sigma,\bar{\sigma}$ of non-self dual residuals are exactly those that are removed by quotienting out by isomorphism. 
\\

To prove Lemma \ref{lem:can} it suffices that $\on{rank}\on{ker}\wt{\Phi}_\ell^\Z=1$ since, as seen, there is already a cancelling tuple obtained from cycling around the residual quiver and so if the kernel of $\wt{\Phi}_\ell^\Z$ is cyclic then this special cancelling tuple must generate it. This follows because the coordinate vector of this cancelling tuple in the standard basis of $\mathbb{Z}\la\mathcal{R}_\ell/{\cong}\ra$ is primitive as it contains a $1$ as an entry corresponding to the single occurence of the self-dual indecomposable singularity $\frac{1}{\ell}(1,1)$ when $\ell$ is odd or $\frac{1}{2\ell}(1,1)$ when $\ell$ is even. Rank-nullity then informs us that Lemma \ref{lem:can} is equivalent to
$$S(\ell)=\phi(\ell)-2$$
or, equivalently, that there are exactly two self-dual singularities contained in $\mathcal{R}_\ell$ for any $\ell$.
\\

A self-dual residual of width $w$ and slope $c$ is one for which $(wc-1)^2\equiv 1\on{mod}{\ell w}$ or, equivalently, $wc\equiv 2\on{mod}{\ell}$. Suppose $\ell$ is odd. There is then exactly one self-dual residual of width $1$ and $2$ given by the equation $w=2\bar{c}$ as $2$ is invertible modulo $\ell$. If $\ell$ is even then at width $w=2$ one can solve for invertible $c$ obtaining $c=1$. Indeed $c\equiv 1\on{mod}{\ell}$ is needed but this satisfies the coprimality conditions.

\begin{lemma} There are at most two self-dual indecomposables at any local index.
\end{lemma}

\begin{proof}
Consider the residual quiver $\mathcal{Q}(\ell)$ for local index $\ell$, which is a $\phi(\ell)$-cycle. It carries an involution given by $\iota:\sigma\mapsto\bar{\sigma}$ which reverses the direction of the arrows. $\iota$ hence fixes at most $2$ vertices, which correspond to self-dual indecomposables by definition.
\end{proof}

There are actually exactly two self-dual residuals at every local index; explicitly these are
$$\begin{cases}
\frac{1}{\ell}(1,1),\;\;\;\frac{1}{2\ell}(1,1) & \text{if $\ell$ is odd,} \\
\frac{1}{2\ell}(1,1),\;\;\;\frac{1}{2\ell}(1,\ell+1) & \text{if $\ell\equiv 0,4\on{mod}{8}$,} \\
\frac{1}{2\ell}(1,1),\;\;\;\frac{1}{4\ell}(1,\ell+1) & \text{if $\ell\equiv 2\on{mod}{8}$,} \\
\frac{1}{2\ell}(1,1),\;\;\;\frac{1}{4\ell}(1,3\ell+1) & \text{if $\ell\equiv 6\on{mod}{8}$.}
\end{cases}$$
as can be verified by some modular arithmetic. This proves Lemma \ref{lem:can} subject to Conjecture \ref{conj:1}.

\begin{cor} One can order the generating set $\mathcal{R}_\ell/{\cong}$ of $\Z\la\mathcal{R}_\ell/{\cong}\ra$ such that in those coordinates $\on{ker}{\wt{\Phi}_\ell^\Z}=\Z\cdot(1,2,...,2,1)$ where the ones correspond to the two self-dual indecomposables and the twos identify the $\frac{1}{2}\phi(\ell)-1$ non-self dual pairs.
\end{cor}

It also follows that the residual quiver always takes the form
\\

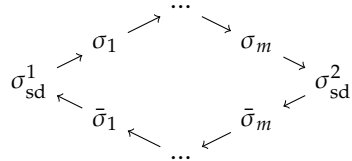
\hspace{2in}
\begin{tikzpicture}[scale=1]
\node (sd1) at (-5,2) {$\sigma_{\text{sd}}^1$};
\node (b) at (-4,2.5) {$\sigma_1$};
\node (e1) at (-3,3) {$...$};
\node (c) at (-2,2.5) {$\sigma_m$};
\node (sd2) at (-1,2) {$\sigma_{\text{sd}}^2$};
\node (cb) at (-2,1.5) {$\bar{\sigma}_m$};
\node(e2) at (-3,1) {$...$};
\node (bb) at (-4,1.5) {$\bar{\sigma}_1$};

\draw[->] (sd1)--(b);
\draw[->] (b)--(e1);
\draw[->] (e1)--(c);
\draw[->] (c)--(sd2);
\draw[->] (sd2)--(cb);
\draw[->] (cb) -- (e2);
\draw[->] (e2) -- (bb);
\draw[->] (bb) -- (sd1);
\end{tikzpicture}

where the $\sigma_{\text{sd}}^i$ are the two self-dual indecomposables, and the $\sigma_i,\bar{\sigma}_i$ are the non-self dual pairs of indecomposables; so $m={\frac{1}{2}\phi(\ell)-1}$. Observe that if one maximally shatters a non-elementary $T$-cone $\tau$ then one must obtain a non-minimal cancelling tuple: the result will consist of a minimal cancelling tuple for each elementary $T$-cone inside $\tau$. Because the degree contribution $A_\tau$ of a $T$-singularity $\tau$ is equal to its width divided by its local index, it follows that:

\begin{cor}\label{cor:pos} If $\sum{Q_{\sigma_i}}=0$ then $\sum{A\sigma_i}\in\mathbb{N}$ and the second sum is zero iff the list of $\sigma_i$s is empty.
\end{cor}

\subsection{Conjecture \#2}

Another natural question is whether or not minimal cancelling tuples can involve singularities of different local indices. This is unresolved but serves to sharpen later results and so is stated as a conjecture.

\begin{conjecture}\label{conj:2} Suppose $\sum_{i=1}^nQ_{\sigma_i}=0$. Then $\ell_{\sigma_i}=\ell_{\sigma_j}$ for all $i,j$.
\end{conjecture}

\begin{remark} Let $\on{CQS}(\ell)$ be the set of cyclic quotient singularities of local index $\ell$ and consider the map $\rho:\on{CQS}(\ell)\to\mathbb{N}[\mathcal{R}_\ell]\to\mathbb{N}[\mathcal{R}_\ell/{\cong}]$. This is a surjection onto a free semigroup of rank $\frac{1}{2}\phi(\ell)+1$. The $T$-singularities form a ray given in coordinates as above by $\N\cdot(1,2,...,2,1)$ within this semigroup that exactly consists of the $\mathbb{Q}$-Gorenstein smoothable singularities. We are curious about any analogous combinatorics in higher dimensions, which may lead to or confirm a suitable definition of `residual singularity' there.
\end{remark}

\section{Recovering baskets from Hilbert series}

\subsection{Decomposing baskets} We will start by decomposing a basket of singularities into two pieces - one containing only cancelling tuples, and one that is actually detectable by the Hilbert series.

\begin{definition} Let $X$ be an orbifold del Pezzo surface with basket $\mathcal{B}_X$.
\begin{itemize}
\item An \textit{invisible basket} in $\mathcal{B}_X$ is a subset $\mathcal{IB}\subset\mathcal{B}_X$ such that $\sum_{\sigma\in\mathcal{IB}}Q_\sigma=0$ and that no nonempty subcollection $\emptyset\not=T\subset\mathcal{B}\setminus\mathcal{IB}$ has that $\sum_{\sigma\in T}Q_\sigma=0$.
\item The collection $\mathcal{B}_X\setminus\mathcal{IB}$ is called the \textit{reduced basket} for $\mathcal{IB}$ in $\mathcal{B}_X$. It will be denoted by $\mathcal{RB}$.
\end{itemize}
\end{definition}

Equivalently, call a multiset $\mathcal{S}$ of singularities \textit{invisible} if $\sum_{\sigma\in\mathcal{S}}Q_\sigma=0$. An invisible basket for $X$ is a maximal invisible submultiset $\mathcal{IB}\subset\cB_X$.

\begin{definition} Let $\mathcal{B}$ be a multiset of singularities. Set $\mathcal{B}(\ell):=\{\sigma\in\mathcal{B}:\ell_\sigma=\ell\}$ to be the $\ell$\textit{th piece} of $\mathcal{B}$. Define $\mathcal{RB}(\ell)$ and $\mathcal{IB}(\ell)$ similarly.
\end{definition}

The orbifold correction terms of a Hilbert series provide data only on the level of a reduced basket as an invisible basket is by definition invisible to it. The extent to which a series determines a reduced basket is discussed below. Conjecture \ref{conj:2} implies the following:

\begin{quote}\textit{Suppose $X$ is an orbifold del Pezzo surface with basket $\mathcal{B}$ featuring singularities of local indices $\ell_1,..,\ell_N$. Then the decomposition
$$\on{Hilb}_X(t)=\frac{1+(K_X^2-2)t+t^2}{(1-t)^3}+\sum_{\sigma\in\mathcal{B}(\ell_1)}Q_\sigma+...+\sum_{\sigma\in\mathcal{B}(\ell_N)}Q_\sigma$$
is unique in that it corresponds to grouping terms with a common denominator of the form $1-t^{\ell_i}$.}
\end{quote}

Consequently, write $Q_\mathcal{B}(\ell):=\sum_{\sigma\in\mathcal{B}(\ell)}{Q_\sigma}$ for the $\ell$th part of the decomposition of the orbifold contribution from $\mathcal{B}$. An easy fact independent of this conjecture is that the initial term can be identified from the Hilbert series as a whole as the only part with a triple pole at $1$. Note that the order of vanishing at $1$ cannot be diminished by the numerator since $K_X^2>0$ implies that $1+(K_X^2-2)t+t^2$ cannot have $1$ as a root.

\begin{cor}\label{cor:sdeg} The degree of an orbifold del Pezzo surface is determined by its Hilbert series.
\end{cor}

\subsection{Convex geometry and reduced baskets}

We situate the problem of computing the possible reduced baskets for an orbifold del Pezzo surface with a given Hilbert series or total orbifold contribution in the setting of convex geometry where it is most easily visualised. This will be accompanied by an example for local index $5$. Recall that $\on{Res}(\ell)$ is the set of indecomposable singularities of local index $\ell$.
\\

Denote by $\on{Res^+}(\ell)$ a choice of one residual singularity of each hyperplane inverse pair $\{\sigma,\sigma^{-1}\}$. There is a map
$$\ph:\Z\langle\on{Res^+}(\ell)\rangle\to\N\langle\on{Res}(\ell)\rangle$$
given by interpreting interpreting a linear combination $v_1\sigma_1+\dots+v_r\sigma_r\in\Z\langle\on{Res^+}(\ell)\rangle$ as a basket of singularities by taking $v_i$-copies of $\sigma_i$ if $v_i\geq0$ or $(-v_i)$ copies of $\sigma_i^{-1}$ if $v_i<0$. Observe that by construction the image of $\ph$ consists of all baskets containing no cancelling pairs and hence it must contain every possibility for a reduced basket. Composing with the map $\Phi_\ell$ to $\Delta(\ell)$ gives the linear map $\Phi_\ell^+:\Z\langle\on{Res^+}(\ell)\rangle\to\Delta(\ell)$ associating to $\sigma$ the $\delta$-vector of its orbifold contribution $Q_\sigma$. Note that this is well-defined as $Q_{\sigma^{-1}}=-Q_\sigma$. The elements of $\on{ker}{\Phi_\ell^+}$ correspond to cancelling $m$-tuples with $m>2$.

\begin{example}\label{ex:l51} At local index $5$ there are eight residual singularities falling into the following inverse pairs:
$$\{\frac{1}{5}(1,1),\frac{1}{20}(1,11)\},\{\frac{1}{5}(1,2),\frac{1}{20}(1,3)\},\{\frac{1}{10}(1,1),\frac{1}{15}(1,11)\},\{\frac{1}{10}(1,3),\frac{1}{15}(1,2)\}$$
Making the choice
$$\on{Res^+}(5)=\{\sigma_1=\frac{1}{5}(1,1),\sigma_2=\frac{1}{20}(1,3),\sigma_3=\frac{1}{10}(1,1),\sigma_4=\frac{1}{15}(1,2)\}$$
gives orbifold contributions with $\delta$-vectors
$$q_1=(1,-2,1),q_2=(2,1,2),q_3=(3,4,3),q_4=(1,3,1)$$
Their span is a two dimensional lattice since $q_1+q_4=q_2$ and $q_1+2q_4=q_3$. These relations define the cancelling tuples
$$\frac{1}{5}(1,1),\frac{1}{5}(1,2),\frac{1}{15}(1,2)\;\;\text{ and }\;\;\frac{1}{5}(1,1),\frac{1}{15}(1,11),\frac{1}{15}(1,2),\frac{1}{15}(1,2).$$
\end{example}

Now suppose that $\delta\in\Delta(\ell)$ is the $\delta$-vector of some rational function
$$Q=\frac{\delta_1t+\delta_2t^2+\dots+\delta_{\ell-2}t^{\ell-2}}{\ell(1-t^\ell)}$$
that could be the total orbifold contribution of some basket $\cB$ of singularities of local index $\ell$. Suppose $\cB_0\in\Z\la\on{Res^+}(\ell)\ra$ satisfies
$$\Phi_\ell^+(\cB_0)=\delta$$
That is, it is a particular solution to the problem of finding a basket producing the given total orbifold contribution. By definition, any other basket with this property will differ as an element of $\Z\la\on{Res^+}(\ell)\ra$ by a cancelling tuple or, equivalently, an element of $\on{ker}\Phi_\ell^+$.

\begin{example}\label{ex:l52} Consider the orbifold contribution
$$Q=\frac{2t^3+t^2+2t}{5(1-t^5)}$$
with $\delta$-vector $(2,1,2)$. A particular basket producing this orbifold contribution is
$$\cB_0=\{\frac{1}{5}(1,1),\frac{1}{15}(1,2)\}$$
corresponding to the vector $(1,0,0,1)\in\Z^4$ in the coordinates of Example \ref{ex:l51}. The set of baskets containing no cancelling pairs with this total orbifold contribution is, in coordinates,
$$\cB_0+\on{ker}\Phi_\ell^+=\{(1+\lambda+\mu,-\lambda,-\mu,1+\lambda+2\mu):\lambda,\mu\in\Z\}$$
\end{example}

In order to find all of the reduced baskets - those not containing any cancelling tuples - that produce a given total orbifold contribution, one must exclude all baskets containing cancelling tuples. To this end, define the \textit{signature} of a vector $v\in\Z^n$ to be
$$\on{sgn}(v):=(\on{sgn}(v_1),\dots,\on{sgn}(v_n))$$
where $\on{sgn}$ is the usual sign function satisfying $\on{sgn}(0)=0$. Define
$$L_v:=\bigoplus_{v_i\not=0}\N\cdot\on{sgn}(v_i)e_i\oplus\bigoplus_{v_i=0}\Z\cdot e_i\text{ and }S(v):=v+L_v$$
A vector $u\in\mathbb{Z}^n$ is said to \textit{feature} in another vector $v\in\mathbb{Z}^n$ if $u\in S(v)$. Note that $S(v)$ is a smooth affine rational polyhedral cone. Inside the lattice $\Z\la\on{Res^+}(\ell)\ra$ using as coordinates the distinguished basis $\on{Res^+}(\ell)$, the cone $S(v)$ consists of the baskets containing $v$ since allowing no sign changes corresponds to the property that no singularities appearing in $v$ can be removed in moving to a basket found in $S(v)$.
\\

If $v$ is a cancelling tuple, no reduced baskets can lie in $S(v)$ as all baskets there will all contain the cancelling tuple $v$. In particular, there can only be finitely many solutions along any affine ray of the form $\{u+\lambda v:\lambda\geq0\}$ parallel to $\on{ker}\Phi_\ell^+$, since eventually the cancelling tuple $v\in\on{ker}\Phi_\ell^+$ will feature in $u+\lambda v$ for $\lambda\gg0$. This shows...

\begin{lemma}\label{lem:ray} There only finitely many reduced baskets along each affine ray parallel to $\on{ker}\Phi_\ell^+$.
\end{lemma}

Working with a given $\delta$-vector $\delta\in\Delta(\ell)$ and particular choice of basket $\cB_0$ whose total orbifold contribution has this $\delta$-vector, this means that the reduced baskets producing this total orbifold contribution biject with lattice points in the complement of the union as $v$ ranges over all cancelling tuples of the convex rational polyhedra
$$K_{v,\delta}:=S(v)\cap(\cB_0+\on{ker}\Phi_\ell^+)=S(v)\cap(\Phi_\ell^+)^{-1}(\delta)$$
inside $\cB_0+\on{ker}\Phi_\ell^+$. More concisely, the reduced baskets producing an orbifold contribution with $\delta$-vector $\delta$ biject with lattice points inside
$$(\Phi_\ell^+)^{-1}(\delta)\setminus\bigcup_{v\in\on{ker}\Phi_\ell^+}K_{v,\delta}$$

\begin{example}\label{ex:l53} Returning to Example \ref{ex:l52} and the $\delta$-vector $(2,1,2)\in\Delta(5)$, since $\on{ker}\Phi_\ell^+\cong\Z^2$ one can sketch the intersection $K_{v,\delta}$ of the cones $S(v)$ with $\cB_0+\on{ker}\Phi_\ell^+$. Here some of the resulting polyhedra $K_{v,\delta}$ are drawn, enough for the purposes at hand.

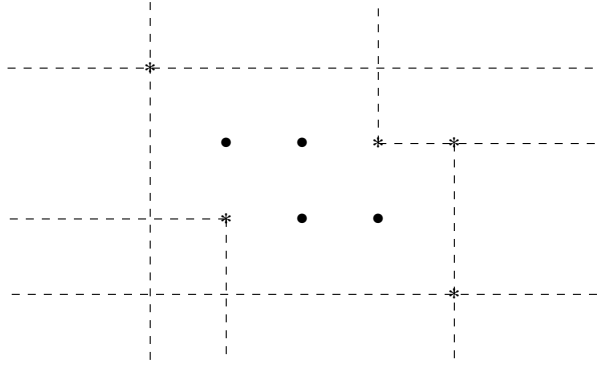
\begin{center}
\begin{tikzpicture}
\node (o) at (0,0){};
\node (e1) at (4,0){};
\node (e2) at (0,2){};
\node (-e1) at (-4,0){};
\node (-e2) at (0,-3){};

\node (b1) at (0,0){$\bullet$};
\node (b2) at (0,-1){$\bullet$};
\node (b3) at (1,-1){$\bullet$};
\node (b4) at (-1,0){$\bullet$};

\node (c1) at (-2,1){$*$};
\node (c2) at (-1,-1){$*$};
\node (c3) at (1,0){$*$};
\node (c4) at (2,-2){$*$};
\node (c5) at (2,0){$*$};

\node (d1) at (-2,2){};
\node (d2) at (-4,1){};
\node (d3) at (4,1){};
\node (d4) at (-2,-3){};
\node (d5) at (1,2){};
\node (d6) at (2,-3){};
\node (d7) at (-4,-1){};
\node (d8) at (-1,-3){};
\node (d9) at (4,-2){};
\node (d10) at (-4,-2){};

\draw[dashed] (d1) to (d4);
\draw[dashed] (d2) to (d3);
\draw[dashed] (c3.center) to (d5);
\draw[dashed] (c3.center) to (e1);
\draw[dashed] (c5.center) to (d6);
\draw[dashed] (c2.center) to (d7);
\draw[dashed] (c2.center) to (d8);
\draw[dashed] (d9) to (d10);
\end{tikzpicture}
\end{center}

In this example the polyhedra $K_{v,\delta}$ exclude a cobounded set and so there is only a finite number of possible reduced baskets for the $\delta$-vector $(2,1,2)$. That is, there are only finitely many reduced baskets giving rise to the given orbifold contribution. They can be seen from this to be
$$(1,0,0,1)\;\;\;(1,-1,1,0)\;\;\;(0,0,1,-1)\;\;\;(0,1,0,0)$$
$$\text{or }\frac{1}{5}(1,1),\frac{1}{15}(1,2);\;\;\;\frac{1}{5}(1,1),\frac{1}{5}(1,1),\frac{1}{5}(1,2),\frac{1}{10}(1,1);\;\;\;\frac{1}{10}(1,1),\frac{1}{10}(1,3);\;\;\;\frac{1}{20}(1,3).$$
They all have total degree contribution $\sum_{\sigma\in\mathcal{RB}}A_\sigma=-8/5$. 
\end{example}

\subsection{Proof of main result}

We return to the situation of general but fixed local index $\ell$ in order to generalise the phenomena found in the example above. We apply Conjecture \ref{conj:1} to show that there are only a finite number of choices for reduced basket given a particular total orbifold contribution.

\begin{thm}\label{thm:fin} There are only finitely many reduced baskets for a given total orbifold contribution.
\end{thm}

\begin{proof} From previous discussion around Conjecture \ref{conj:2} it suffices that we consider the case of a single local index $\ell$. Recall the maximal shattering map $\rho:\N\la\on{Res}(\ell)\ra\to\N\la\mathcal{R}_\ell\ra$. Let $v$ be the image of an elementary $T$-singularity under $\rho$. As seen previously, in special coordinates this is the vector $(1,2,\dots,2,1)$.
\\

Suppose that $\cB$ and $\cB'$ are two baskets with the same total orbifold contribution. Then, subject to Conjecture \ref{conj:1}, their maximal shatterings $\rho(\cB)$ and $\rho(\cB')$ must differ by a vector of the form $\lambda\cdot v$ in coordinates as above. Notice also that there are only finitely many baskets $\cB\in\N\la\on{Res}(\ell)\ra$ with a given maximal shattering $\rho(\cB)=\mathcal{T}$ as there are only finitely many ways to glue together the finite number of singularities in $\mathcal{T}$. Hence, if for any given $\mathcal{T}\in\N\la\mathcal{R}_\ell\ra$ there is some $\lambda\gg 0$ such that any $\cB\in\N\la\on{Res}(\ell)\ra$ with $\rho(\cB)=\mathcal{T}+\lambda\cdot v$ contains a cancelling tuple, then the result would be shown. The reduced baskets with the total orbifold contribution for $\mathcal{T}$ would then correspond to the preimages under $\rho$ of $\mathcal{T}+\mu\cdot v$ with $0\leq\mu<\lambda$ and that contain no cancelling tuples.
\\

This statement is equivalent to the following, which we will actually prove: for each $\mathcal{T}\in\N\la\mathcal{R}_\ell\ra$ and for each $N\in\N$ there is $\lambda\gg 0$ such that every $\cB\in\N\la\on{Res}(\ell)\ra$ with $\rho(\cB)=\mathcal{T}+\lambda\cdot v$ contains at least $N$ cancelling tuples.
\\

We proceed by induction on $|\mathcal{T}|=n$. If $n=0$, then one can construct arbitrarily many cancelling tuples inside baskets whose maximal shattering is of the form $\lambda\cdot v$ as follows. Since the indecomposable singularities in a maximal shattering of the form $\lambda\cdot v$ coalesce to form a $T$-singularity $\tau$ of width $\lambda$, a basket $\cB$ of size $p$ with this maximal shattering corresponds to a choice of $p-1$ lattice points on the edge of $\tau$, for which the corresponding crepant blowups produce the singularities in $\cB$. These lattice points must be width less than $\ell$ apart in order for the corresponding singularities to be residual. Notice that $p-1\geq\lambda$ as one has to choose a lattice point inside each of the elementary $T$-singularities constituting $\tau$. If $\lambda=\ell+1$ then at least two of the lattice points must differ by a cone of width a multiple of $\ell$, which is hence a $T$-singularity. The cones subtended by the lattice points between these two lattice points thus form a cancelling tuple. Repeating the process starting from the end of this $T$-singularity, one can produce $N$ cancelling tuples by setting $\lambda=N(\ell+1)$.
\\

Suppose the statement is true for all $\mathcal{T}$ of size $n-1$. Let $\mathcal{T}$ have size $n$ and choose $\sigma_0\in\mathcal{T}$. Then $\mathcal{T}_0=\mathcal{T}\setminus\{\sigma_0\}$ has size $n-1$ and so there is $\lambda$ such that every $\cB$ with $\rho(\cB)=\mathcal{T}_0+\lambda\cdot v$ contains at least $N+1$ cancelling tuples. By adding the single singularity $\sigma_0$ back in to such a basket $\cB$, one can reduce the number of cancelling tuples at most by $1$. Thus, this same $\lambda$ has the property that every $\cB$ with $\rho(\cB)=\mathcal{T}+\lambda\cdot v$ contains at least $N$ cancelling tuples.
\end{proof}

Using the recursion for $\lambda$ from the induction, one obtains...

\begin{cor} For a collection of indecomposable singularities $\mathcal{T}$ of size $n$, every basket $\cB$ with maximal shattering $\rho(\cB)=\mathcal{T}+\lambda\cdot v$ contains at least $N$ cancelling tuples when $\lambda\geq(N-n)(\ell+1)$.
\end{cor}

Though the proof above relies on Conjecture \ref{conj:1}, it seems plausible that there is a direct convex geometric proof bypassing this dependency, the most general form of which could look like the following. Recall that an affine ray is a subset of $\R^n$ of the form $\{u+\lambda v:\lambda\geq0\}$ for some $u,v\in\R^n$ and for some $\lambda_0\in\R_{\geq0}$. It is rational if $u$ and $v$ can be chosen to be lattice points.

\begin{conjecture} Suppose $K=\bigcup_{i=1}^NK_i\subset\R^n$ is the union of finitely many affine convex rational polyhedra and suppose that $\R^n\setminus K$ contains infinitely many lattice points. Then $\R^n\setminus K$ contains an affine rational ray.
\end{conjecture}

In combination with Lemma \ref{lem:ray}, by choosing $K_i$ appropriately amongst the $K_{v,\delta}$ one obtains the same result but independently of Conjecture \ref{conj:1}. The pertinent set
$$\on{RBod}(\ell,\delta):=(\Phi_\ell^+)^{-1}(\delta)\setminus\bigcup_{v\in\on{ker}\Phi_\ell^+}K_{v,\delta}$$
is then a bounded subset of $\Z\la\on{Res^+}(\ell)\ra$ whose lattice points correspond to the finite number of reduced baskets whose total orbifold contribution has $\delta$-vector $\delta$. We call this subset the \textit{reduced body} for the given total orbifold contribution. 
\\

Define the \textit{width} of a collection of singularities to be the sum of the widths of the singularities. If the collection arises from shattering a $T$-singularity $\tau$ then its width equals the width of $\tau$, which is also related to its total contribution to the degree via Corollary \ref{cor:can}.
\\

As noted in Corollary \ref{cor:sdeg} the degree of an orbifold del Pezzo surface $X$ can be read off from its Hilbert series. It has a decomposition
$$K_X^2=12-(\mathcal{RK}_X^2(\ell_1)+...+\mathcal{RK}^2(\ell_N)+\mathcal{IK}^2)$$
where $\mathcal{RK}^2(\ell):=\sum_{\sigma\in\mathcal{RB}(\ell)}{A_\sigma}$ is the degree contribution from the $\ell$th piece of the reduced basket and $\mathcal{IK}^2$ is the (nonnegative integral) contribution from the \textit{extended invisible basket}
$$\widehat{\mathcal{IB}}:=\mathcal{IB}\cup\{\text{$T$-singularities on $X$}\}$$
with
$$\mathcal{IK}^2=\sum_{\sigma\in\widehat{\mathcal{IB}}}A_\sigma=\sum_{\sigma\in\widehat{\mathcal{IB}}}\frac{\on{width}(\sigma)}{\ell_\sigma}$$
from Corollary \ref{cor:can}. Notice that the Euler number term from Lemma \ref{lem:degf} is accounted for in $\mathcal{IK}_X^2$ by including the $T$-singularities. Knowing a reduced basket for $X$ then prescribes the degree contribution of the corresponding invisible basket. Notice that the definitions of these quantities only rely on the basket of singularities on $X$ with a choice of invisible/reduced basket and so it makes sense to speak of each of them as associated to just a basket with a choice of invisible/reduced basket.

\begin{example}\label{ex:l54} For the total orbifold contribution from Example \ref{ex:l53} with $\delta$-vector $(2,1,2)\in\Delta(5)$ the four possible reduced baskets all have $\mathcal{RK}^2=-8/5$. Suppose one fixes a Hilbert series $H(t)$ with this total orbifold contribution from which one can read the degree $K^2$. Any invisible basket must have total degree contribution $\mathcal{IK}^2=12-\mathcal{RK}^2-K_X^2=\frac{68}{5}-K_X^2$, which is indeed integral since $K_X^2\equiv\frac{3}{5}\on{mod}{\mathbb{Z}}$ for any surface with this Hilbert series.
\end{example}

We now collect the results of the paper subject to Conjectures \ref{conj:1} and \ref{conj:2}, and established by the results and discussion of the last two sections.

\begin{thm}\label{thm:main} Fix a power series $H(t)\in\N\ldbrack t\rdbrack$. Either there are no orbifold del Pezzo surfaces with Hilbert series equal to $H(t)$, or
\begin{itemize}
\item the reduced basket of an orbifold del Pezzo surface with Hilbert series equal to $H(t)$ is one of a finite number of possibilities, which are determined by the orbifold correction part of $H(t)$ and in bijection with the lattice points of the associated reduced body.
\item the basket of such an orbifold del Pezzo surface with reduced basket $\mathcal{RB}$ is given by adding to the singularities in $\mathcal{RB}$ a collection of cancelling tuples - which are obtained by shattering $T$-singularities - whose total degree contribution is determined by $\mathcal{RB}$ and $H(t)$.
\end{itemize}
\end{thm}

\begin{example} For the total orbifold contribution
$$\frac{8t^3-t^2+8t}{5(1-t^5)}$$
there are 18 possible reduced baskets which all lie in the affine plane $(5,0,0,3)+\on{ker}{\Phi}$. Unlike in Example \ref{ex:l53}, the reduced body contains multiple lattice points along a single ray: $(5,0,0,3)+\lambda(1,0,-1,2)$, which have differing degree contributions. This shows that, in general, the total degree contribution from the reduced basket depends on more than simply the orbifold contribution.
\end{example}

To discuss an application of this result, recall that the Gorenstein index $\ell_X$ of a Fano variety $X$ is the smallest positive integer $m$ such that $mK_X$ is Cartier. For orbifold del Pezzo surfaces, this is the lowest common multiple of the local indices of all the singularities on $X$.

\begin{example}\label{ex:l55} Returning to the total orbifold contribution of Example \ref{ex:l54} with $\delta$-vector $(2,1,2)\in\Delta(5)$ it follows from the above that a basket for an orbifold del Pezzo surface $X$ of Gorenstein index $5$ with this orbifold contribution can contain at most $82$ singularities: $4$ at most from a reduced basket and then at most $5+1=6$ in a minimal cancelling tuple from the discussion of the $n=0$ case in the proof of Theorem \ref{thm:fin}, of which there can be at most $13=\lfloor\frac{68}{5}\rfloor$ from the Fano condition $K_X^2>0$. If $X$ is allowed to have Gorenstein index $5\ell$ then imitating this calculation yields that the number of singularities on $X$ can be at most
$$4+13(5\ell+1)=65\ell+17$$
\end{example}

This example generalises easily using Theorem \ref{thm:main} to the following corollary.

\begin{cor}\label{cor:bdd} Fix a total orbifold contribution $Q\in\Q(t)$ and a positive integer $\ell_*$. There exists a number $N(Q,\ell_*)$ dependent only on and computable from $Q$ and $\ell_*$ such that any orbifold del Pezzo surface with total orbifold contribution $Q$ and Gorenstein index bounded above by $\ell_*$ has at most $N(Q,\ell_*)$ singularities.
\end{cor}

As in the example, it is straightforward to compute $N(Q,\ell_*)$ using the proof of Theorem \ref{thm:fin} once the reduced bodies for the $\ell$th pieces of $Q$ have been found.

\subsection{Degree bounds}

Let $H(t)\in\N\ldbrack t\rdbrack$. If $H(t)$ is the Hilbert series of an orbifold del Pezzo surface $X$, then $X$ must have degree given by the formula in Lemma \ref{lem:hilb}. Denote this number by $K_H^2$. As seen in the previous section, choosing a reduced basket to capture the total orbifold contribution of $H(t)$ enforces a choice of what the degree contribution of a corresponding invisible basket is.
\\

More precisely, for a choice of reduced basket $\mathcal{RB}$ and invisible basket $\mathcal{IB}$ one requires
$$K_H^2=12-\mathcal{RK}^2-\sum_{\sigma\in\widehat{\mathcal{IB}}}\frac{\on{width}(\sigma)}{\ell_\sigma}$$
where $\mathcal{RK}^2:=\sum_{\sigma\in\mathcal{RB}}{A_\sigma}$. Since the last term is nonnegative, one must have $K_H^2+\mathcal{RK}^2\leq12$. There are three cases:
\begin{itemize}
\item if $K_H^2+\mathcal{RK}^2<12$ then there are infinitely many possibilities for the extended invisible basket of an orbifold del Pezzo surface with Hilbert series $H(t)$
\item if $K_H^2+\mathcal{RK}^2=12$, then there is only one possible extended invisible basket for an orbifold del Pezzo surface with Hilbert series $H(t)$ containing $\mathcal{RB}$, which is the empty set since there is no more flexibility in the degree allowing one to add $T$-singularities or cancelling tuples
\item If $K_H^2+\mathcal{RK}^2>12$ then there are no possible extended invisible baskets for orbifold del Pezzo surfaces surfaces with Hilbert series $H(t)$.
\end{itemize}
Testing across all reduced baskets gives the following non-existence result. 

\begin{cor}\label{cor:nonex} With notation as above, if $K_H^2+\mathcal{RK}^2>12$ for all reduced baskets associated to $H(t)$ then there are no orbifold del Pezzo surfaces with Hilbert series $H(t)$.
\end{cor}

For example there are no orbifold del Pezzo surfaces with Hilbert series
$$\frac{1+mt+t^2}{(1-t)^3},m\geq11$$
since the only reduced basket for this power series is the empty set with $\mathcal{RK}^2=0$ and so $K_H^2+\mathcal{RK}^2=m+2>12$. Observe that in addition there are no toric orbifold del Pezzo surfaces with Hilbert series
$$\frac{1+10t+t^2}{(1-t)^3}$$
since a projective toric surface has at least three (possibly smooth) affine pieces or singularities corresponding to the faces of its polygon and so must feature at least one cancelling tuple.
\\

In general, as there are only finitely many reduced baskets for a given total orbifold contribution, the discussion above along with the Fano condition $K_X^2>0$ show how to produce bounds on the degree of any orbifold del Pezzo surface with that total orbifold contribution.

\begin{thm}\label{thm:deg} For a given collection of residual singularities $\cB$ there exist constants $m>0$ and $M>0$ dependent only on and computable from $\cB$ such that, for any orbifold del Pezzo surface $X$ with $\cB_X=\cB$,
$$m\leq K_X^2\leq M$$
\end{thm}

\subsection{A different perspective}

Broadly speaking, the approach taken so far in this paper has been deformation-theoretic: we have sought to classify the possible collections of singularities that could correspond to baskets of singularities (which are by construction deformation classes of singularities) on an orbifold del Pezzo surface with a given Hilbert series. An alternative perspective one could take is to allow not just deformations but also crepant blowups. From this perspective, residual singularities are replaced by indecomposable singularities as the appropriate notion of `rigid' singularities.
\\

To briefly explore this perspective, define two invariants $\Psi_\ell(X)$ and $\wt{\Psi}_\ell(X)$ of an orbifold del Pezzo surface $X$ by
$$\Psi_\ell(X):=\rho(\cB_X(\ell))\in\N\la\mathcal{R}_\ell\ra\text{ and }\wt{\Psi}_\ell(X):=\rho(\wt{\cB}_X(\ell))\in\N\la\mathcal{R}_\ell\ra$$
That is, $\Psi_\ell(X)$ is the collection of indecomposable singularities obtained by maximally shattering all the residues of singularities of local index $\ell$ on $X$, and $\wt{\Psi}_\ell(X)$ is the collection of indecomposable singularities obtained by maximally shattering all singularities on $X$. The former allows both deformations and crepant blowups to be taken; the latter allows no deformation, for example, it recognises the $T$-singularities on $X$. In this language, Theorem \ref{thm:main} becomes the following.

\begin{thm} Fix a power series $H(t)\in\N\ldbrack t\rdbrack$. There exists an affine ray $\varrho_{H(t)}\subset\N\la\mathcal{R}_\ell\ra$ such that any orbifold del Pezzo surface $X$ with Hilbert series $H(t)$ has the property that $\Psi_\ell(X)$ lies on $\varrho_{H(t)}$. Moreover the slope of $\varrho_{H(t)}$ is independent of $H(t)$ and corresponds to a maximally shattered elementary $T$-singularity of local index $\ell$. The same is true for $\wt{\Psi}_\ell(X)$ with the same ray $\varrho_{H(t)}$ (possibly truncated).
\end{thm}

The theorem also holds if the Hilbert series $H(t)$ is replaced by a total orbifold contribution $Q\in\Q(t)$.

\end{document}